\documentclass{amsart}
\usepackage{amsmath}
\usepackage{amssymb}
\usepackage[all]{xy}
\usepackage{enumerate}
\numberwithin{equation}{section}

\theoremstyle{plain}
\newtheorem{thm}{Theorem}[section]
\newtheorem{lemma}[thm]{Lemma}

\newtheorem{prop}[thm]{Proposition}

\theoremstyle{definition}

\newtheorem{remark}[thm]{Remark}

\begin{document}

\title{On Liftings of Projective Indecomposable $G_{(1)}$-Modules}

\author[Paul Sobaje]
{Paul Sobaje}

\begin{abstract}
Let $G$ be a simple simply connected algebraic group over an algebraically closed field $k$ of characteristic $p$, with Frobenius kernel $G_{(1)}$.  It is known that when $p\ge 2h-2$, where $h$ is the Coxeter number of $G$, the projective indecomposable $G_{(1)}$-modules (PIMs) lift to $G$, and this has been conjectured to hold in all characteristics.  In this paper, we explore the lifting problem via extensions of algebraic groups, following the work of Parshall and Scott who in turn build upon ideas due to Donkin.  We prove various results which augment this approach, and as an application demonstrate that the PIMs lift to $G_{(1)}H$, for particular closed subgroups $H \le G$ which contain a maximal torus of $G$.
\end{abstract}

\maketitle

\section{Introduction}

Let $G$ be a simple simply connected algebraic group over an algebraically closed field $k$ of characteristic $p$, with Frobenius kernel $G_{(1)}$.  A conjecture by Humphreys and Verma, at this point more than 40 years old, states that the projective indecomposable $G_{(1)}$-modules (PIMs) should lift to $G$.  Donkin later added to this by conjecturing that each PIM should come via restriction of a certain indecomposable tilting module for $G$.  These conjectures are known to be true if $p\ge 2h-2$, where $h$ is the Coxeter number of $G$, but remain open for general $p$ (see Section 2.2).

Parshall and Scott advanced this theory in recent years, proving a stable version of the Humphreys-Verma conjecture \cite{PS2}.  That is, for a given projective indecomposable $G_{(1)}$-module $Q$, they show that there is an integer $n \ge 1$ such that $Q^{\oplus n}$ has a $G$-structure.  Moreover, they detail in \cite{PS} how the lifting problem can be transformed into a problem about algebraic group extensions of $G$ by a connected unipotent group $U_A$, building on constructions due to Donkin \cite{D1}.

In this paper, we seek to better understand the lifting problem from the group extension point of view.  In Section 2 we recall various results pertaining to group extensions and cohomology, in many instances presenting them in a manner tailored to our specific needs.  We then prove several results in Sections 3 and 4 which augment the group extension approach to the lifting problem.

It is well known in all characteristics that the PIMs lift to $G_{(1)}T$, where $T$ is a maximal torus of $G$.  In Section 5, we show that the results in Sections 3 and 4 can be used to improve upon this.  First, we show that the PIMs lift to the subgroup scheme $G_{(1)}N_{W^{\prime}}$ where $N_{W^{\prime}}$ is a subgroup of the normalizer group $N_G(T)$ for which the finite group $N_{W^{\prime}}/T$ has order relatively prime to $p$.  In particular, if $p$ does not divide the order of the Weyl group of $G$, then the PIMs lift to $G_{(1)}N_G(T)$ (over which they are still injective modules).  Second, we show that if $G=SL_n$ and $p \ge n-2$ then the PIMs can be lifted to $G_{(1)}L$-modules when $L$ is a rank $1$ Levi subgroup of $G$.

These applications fall far short of the ultimate goal of lifting the PIMs to $G$, but it should be noted that progress of any kind on this problem has proven quite elusive over the last few decades.  Moreover, we believe that the methods considered are of theoretical interest in their own right.

We also note that the lifting conjecture is often formulated more generally for an arbitrary Frobenius kernel $G_{(r)}$ of $G$, where $r \ge 1$.   However, as observed by Jantzen in \cite[II.11.16 Remark (2)]{J1}, if the PIMs for $G_{(1)}$ lift to $G$, then the PIMs for $G_{(r)}$ lift as well.

In addition to the aforementioned works by Parshall and Scott \cite{PS} and Donkin \cite{D1}, our results and methods most heavily depend upon the papers by McNinch (\cite{M} and \cite{M2}) and Stewart \cite{St} on cohomology and group extensions, and the computations of Bendel, Nakano, and Pillen \cite{BNP} on second cohomology groups.

\subsection{Acknowledgments} We wish to thank Eric Friedlander, Dan Nakano, George McNinch, and Jens Jantzen for helpful discussions and correspondences.


\section{Preliminaries/Recollections}

\subsection{Notation} Throughout $G$ will be a simple simply connected algebraic group over $k$.  We fix a maximal torus $T \le G$, and let $N_G(T)$ be the normalizer of $T$ in $G$.  The Weyl group $W$ denotes the finite group $N_G(T)/T$, while $\Phi$ is the root system of $G$ with respect to $T$.  For each $\alpha \in \Phi$ we fix a root homomorphism $\varphi_{\alpha}: \mathbb{G}_a \rightarrow G$, the image of which we denote by $U_{\alpha}$.  We fix a Borel subgroup $B \le G$ so that the positive roots $\Phi^+$ are those $\alpha$ for which $U_{\alpha} \le B$.  This also determines our choice of simple roots $\Pi \subseteq \Phi$.  For each $I \subseteq \Pi$, we let $L_I$ denote the associated standard Levi subgroup.

We let $X(T)$ be the character group of $T$, $X(T)_+$ the set of dominant weights, and $X_1(T)$ the set of $p$-restricted weights.  Let $\rho = 1/2(\sum_{\alpha \in \Phi^+} \alpha)$.  The action of $W$ on $T$ induces an action on $X(T)$, and for each $w \in W$ and $\lambda \in X(T)$ we will denote this action by $w\lambda$ or $w(\lambda)$.  The dot action is $w \cdot \lambda = w(\lambda + \rho) - \rho$.  The element $w_0 \in W$ is the unique one with the property that $w_0(\Phi^+) = -\Phi^+$.

We will be working within the category of affine group schemes over $k$.  If $H$ is such a group scheme, then it is a functor which for each commutative $k$-algebra $R$ returns the group $H(R)$.  The coordinate algebra of $H$ is denoted $k[H]$.  When we refer to $H$ as an algebraic group, this implies that $k[H]$ is reduced (nilpotent-free).  In such cases, we might write $h \in H$ as shorthand for $h \in H(k)$.

If $H$ is an algebraic group, we denote its Lie algebra by $\mathfrak{h}$.  It is a restricted Lie algebra with restricted enveloping algebra $u(\mathfrak{h})$.  We let $H_{(r)}$ denote the $r$-th Frobenius kernel of $H$.  There is an equivalence between representations for $H_{(1)}$ and $u(\mathfrak{h})$, and $\text{Hom}_k(k[H_{(1)}],k) \cong u(\mathfrak{h})$ \cite[I.9.6(4)]{J1}.

As each $G_{(r)}$ is a normal in $G$, for any closed subgroup scheme $H \le G$ we can form the subgroup scheme $G_{(r)}H \le G$ \cite[I.6.2]{J1}.  In particular we will consider this for $H$ an algebraic subgroup, and will $G_{(r)}H$ an \textit{infinitesimal thickening} of $H$.


We set our notation for important classes of modules, the notation exactly following that given in \cite{J1}.  For every $\lambda \in X(T)_+$ there is a unique simple $G$-module with highest weight $\lambda$ which we denote by $L(\lambda)$.  For every $\lambda \in X(T)$ there is a unique simple $G_{(1)}T$-module $\widehat{L}_1(\lambda)$ with highest weight $\lambda$, this module remains simple upon restriction to $G_{(1)}$ over which we denote it as $L_1(\lambda)$.  The simple modules for $G$ and $G_{(1)}T$ are indexed by the sets with $X(T)_+$ and $X(T)$ respectively, while the simple $G_{(1)}$-modules are in bijection with $X_1(T)$, which is itself in bijection with $X(T)/pX(T) \cong X(T_{(1)})$ via the natural map from $X(T)$ to $X(T_{(1)})$.  Every $\lambda \in X(T)$ can be uniquely written as $\lambda = \lambda_0 + p\mu$, where $\lambda_0 \in X_1(T)$ and $\mu \in X(T)$.  If we let $p\mu$ also denote the one-dimensional $G_{(1)}T$-module on which $G_{(1)}$ acts trivially and $T$ acts via $p\mu$, we have a $G_{(1)}T$-isomorphism $\widehat{L}_1(\lambda) \cong \widehat{L}_1(\lambda_0) \otimes p\mu$. 

For each $\lambda \in X_1(T)$, we let $Q_1(\lambda)$ denote the projective indecomposable $G_{(1)}$-module with unique simple quotient $L_1(\lambda)$.  The socle of $Q_1(\lambda)$ is simple, and is also isomorphic to $L_1(\lambda)$.  Let $\widehat{Q}_1(\lambda)$ be the $G_{(1)}T$-projective cover of $\widehat{L}_1(\lambda)$.  Then it is also the injective hull of $\widehat{L}_1(\lambda)$, and $\widehat{Q}_1(\lambda) \cong Q_1(\lambda)$ over $G_{(1)}$.  If $\mu$ is a weight of $\widehat{Q}_1(\lambda)$, then it is known that $\lambda - 2(p-1)\rho \le \mu \le 2(p-1)\rho + w_0\lambda$.

In general, for any $\mu \in X(T)$ we write $\widehat{Q}_1(\lambda + p\mu)$ for the $G_{(1)}T$-projective cover of $\widehat{L}_1(\lambda + p\mu)$.  We have $\widehat{Q}_1(\lambda + p\mu) \cong \widehat{Q}_1(\lambda) \otimes p\mu$.

For each $\lambda \in X(T)_+$, $T(\lambda)$ will denote the indecomposable tilting module of highest weight $\lambda$.

\subsection{The Conjectures of Humphreys-Verma and Donkin}

It was first conjectured by Humphreys and Verma in \cite{HV} that there should be a $G$-structure for each $G_{(1)}$-module $Q_1(\lambda)$, at which point the conjectured statement had already been shown to be true for $SL_2$ by Jeyakumar \cite{Je}.  We have recalled above that $Q_1(\lambda)$ can always be lifted (non-uniquely) to $G_{(1)}T$.  If it can further be lifted to $G$, then every $G$-module structure must have a simple socle which is isomorphic to $L(\lambda)$.  Therefore, there is only one $G_{(1)}T$-lift of $Q_1(\lambda)$ which could possibly have a $G$-structure, namely $\widehat{Q}_1(\lambda)$.  So the conjecture can equivalently be stated to say that the $G_{(1)}T$-module $\widehat{Q}_1(\lambda)$ lifts to $G$.

Donkin furthered this conjecture by proving in \cite{D1} that there exists some $G$-module $M$ whose formal character is the same as that of $\widehat{Q}_1(\lambda)$ (a necessary condition for the conjecture to hold).  He later conjectured that one $G$-module structure for $\widehat{Q}_1(\lambda)$ is that of the indecomposable tilting module $T(2(p-1)\rho + w_0(\lambda))$.  We note that neither the Humphreys-Verma conjecture, nor this refinement, say anything about the uniqueness of such a $G$-structure, although uniqueness is known to hold in the cases in which both conjectures have been shown to be true.

When $p \ge 2h-2$, then Donkin's conjecture holds, and this is the unique $G$-structure on $\widehat{Q}_1(\lambda)$.  This result is due to Jantzen \cite[II.11.11]{J1}, who improved an earlier result by Ballard \cite{B} which established lifting when $p \ge 3h-3$.  It is also known that in all characteristics the Humphreys-Verma conjecture holds for $SL_3$, though it is unclear if there are multiple non-isomorphic structures (see end of \cite[II.11.16]{J1}).

Parshall and Scott proved in \cite{PS2} that there is always some integer $n \ge 1$ for which $Q_1(\lambda)^{\oplus n}$ lifts to $G$.

\subsection{From Lifting Representations to Group Extensions}

Parshall and Scott in \cite{PS} show how the lifting problem can be converted into a problem about extensions of algebraic groups, building on constructions due to Donkin.  In fact, the authors develop this theory for a more general class of modules than we consider here.

Let $\lambda \in X_1(T)$, and fix a decomposition of $T(2(p-1)\rho + w_0\lambda)$ into indecomposable $G_{(1)}$-submodules.  It is known that $Q_1(\lambda)$ occurs exactly once as a summand of this decomposition and that its $G_{(1)}$-socle is a $G$-submodule of $T(2(p-1)\rho + w_0\lambda)$ (as follows from \cite[II.11.9(3)]{J1} and a result due to Pillen \cite[Theorem 2.5]{D2}).  In \cite{D1}, Donkin defined a variety morphism
$$\sigma: G \rightarrow GL(Q_1(\lambda)), \quad \sigma(g).v = \text{pr}_{Q_1(\lambda)}(g.v),$$
where $\text{pr}_{Q_1(\lambda)}$ is the projection of $T(2(p-1) + w_0\lambda)$ onto $Q_1(\lambda)$ given by the fixed $G_{(1)}$-decomposition.

Let $A = \text{End}_{G_{(1)}}(Q_1(\lambda))$.  Since $Q_1(\lambda)$ is indecomposable over $G_{(1)}$, $A$ is isomorphic as an algebra to $k \oplus J_A$, where $k$ is the center of $\text{End}_{k}(Q_1(\lambda))$ and $J_A$ consists of nilpotent matrices.  In particular, each $X \in J_A$ annihilates the socle of $Q_1(\lambda)$.  Let $U_A = 1+J_A$.  This is a closed unipotent subgroup of $GL(Q_1(\lambda))$.  Donkin showed that $U_A$ and $\sigma(G)$ generate a closed subgroup $G^* \le GL(Q_1(\lambda))$ which extends the representation of $G_{(1)}$.  Additionally, $U_A$ is normal in $G^*$, and the composite of morphisms
$$G \xrightarrow{\sigma} G^* \rightarrow G^*/U_A$$
is an epimorphism of algebraic groups (note that $\sigma$ is only a variety morphism in general).

Parshall and Scott augmented this approach by creating a somewhat more natural group which is actually an extension of $G$ (as opposed to $G^*$ which is an extension of a group isogenous to $G$).  Specifically, $\sigma$ defines morphisms
$$\gamma: G \times G \rightarrow U_A, \quad \kappa: G \times U_A \rightarrow U_A$$
satisfying various properties which allow for the construction of a group $G^{\diamond}$ which is an extension of $G$ by $U_A$ (see \cite[2.2]{PS}).  In the language of the next subsection, there is a strictly exact sequence
$$1 \rightarrow U_A \rightarrow G^{\diamond} \xrightarrow{\pi} G \rightarrow 1$$
so that $G^{\diamond} \cong G \times U_A$ as a variety.  We note that if $U_A$ is abelian, then because it is also $G$-equivariantly isomorphic to its Lie algebra, the morphism $\gamma$ would represent an element in the rational cohomology group $H^2(G,U_A)$, in which case one would construct $G^{\diamond}$ in the standard way (see the following subsection).  However, Parshall and Scott require a more involved approach due to the fact that $U_A$ is not necessarily abelian.

They define a representation $\sigma^{\diamond}: G^{\diamond} \rightarrow GL(Q_1(\lambda))$ given by $\sigma^{\diamond}(g,u).v=\sigma(g)u.v$.  There is a morphism of affine group schemes
$$j_{G_{(1)}}:G_{(1)} \rightarrow G^{\diamond}$$
so that $\pi \circ j_{G_{(1)}}$ is the identity morphism on $G_{(1)}$.  The subgroup scheme $j_{G_{(1)}}(G_{(1)})$ is normal in $G^{\diamond}$, with $U_A$ acting trivially on it, and there is a group scheme isomorphism
$$G^{\diamond}_{(1)} \cong j_{G_{(1)}}(G_{(1)}) \times U_{A,(1)}.$$
The morphism $\sigma^{\diamond} \circ j_{G_{(1)}}$ gives the original representation of $G_{(1)}$ on $Q_1(\lambda)$, and $\sigma^{\diamond}(G^{\diamond})=G^*$.

\begin{thm}\label{sectiontolift}
The $G_{(1)}$-structure on $Q_1(\lambda)$ lifts to $G$ if and only if there is a morphism of algebraic groups $j: G \rightarrow G^{\diamond}$ extending the morphism $j_{G_{(1)}}$.  If there are two such morphisms, $j_1$ and $j_2$, then they define the same $G$-module (up to isomorphism) if and only if there is some $h \in G^{\diamond}$ such that $j_1(g) = hj_2(g)h^{-1}$ for all $g \in G$.
\end{thm}

\begin{proof}
Corollary 3.7 and Remark 3.8 in \cite{PS}.
\end{proof}

Parshall and Scott observe in \cite[Theorem 3.6]{PS} that the group $G^{\diamond}$ is independent of the choice of the variety morphism $\sigma$.  In particular, we may choose the decomposition of $T(2(p-1)\rho + w_0\lambda)$ in the definition of $\sigma$ to be a $G_{(1)}T$-decomposition, so that a single summand is isomorphic to $\widehat{Q}_1(\lambda)$.  In this case, for all $t \in T, v \in Q_1(\lambda)$, and $w \in T(2(p-1)\rho + w_0\lambda)$ we have
$$t.v = \sigma(t).v, \quad \text{pr}_{Q_1(\lambda)}(t.w) = t.(\text{pr}_{Q_1(\lambda)}(w)),$$
hence for all $g \in G,t \in T$, and $v \in Q_1(\lambda)$,
$$\sigma(gt).v=\text{pr}_{Q_1(\lambda)}(gt.v)=\text{pr}_{Q_1(\lambda)}(g.(t.v))=\sigma(g)\sigma(t).v,$$
and
$$\sigma(tg).v=\text{pr}_{Q_1(\lambda)}(tg.v)=t.(\text{pr}_{Q_1(\lambda)}(g..v))=\sigma(t)\sigma(g).v.$$
We see that $\sigma(tgt^{-1}) = \sigma(t)\sigma(g)\sigma(t)^{-1}$.  We may therefore assume that there is an algebraic section $G \rightarrow G^{\diamond}$ which is a group homomorphism on $T$, and for which the image is stable under conjugation by $T$.

Let $K \le T$ be the kernel of the homomorphism $\sigma\mid_T:T \rightarrow GL(Q_1(\lambda))$.  We then see:

\begin{lemma}\label{kernel}
The kernel of the map
$$\sigma^{\diamond}: G^{\diamond} \rightarrow G^*$$
is $K \times \{1\}$.
\end{lemma}

\begin{proof}
By definition, $\sigma^{\diamond}(g,u)=\sigma(g)u$, hence if $(g,u) \in ker(\sigma^{\diamond})$, then $\sigma(g)=u^{-1}$.  Since $\sigma(g)=1$ for all $g \in K$, we see that $K \times \{1\} \subseteq ker(\sigma^{\diamond})$.  On the other hand, if $(g,u) \in ker(\sigma^{\diamond})$, then $\sigma(g) \in U_A$, so that $g$ maps to $1$ in the composite homomorphism
$$G \xrightarrow{\sigma} G^* \rightarrow G^*/U_A.$$
Since $G$ is a simple algebraic group, the kernel of this morphism is a central subgroup contained in $T$.  But $\sigma\mid_T$ is an algebraic group homomorphism, so that $\sigma(g)$ is a semisimple element in $G^*$, hence we must have $\sigma(g)=1=u$, so that $K \times \{1\} = ker(\sigma^{\diamond})$.
\end{proof}

\subsection{Extensions of Algebraic Groups and Cohomology}

Let $\mathcal{G}$ be a linear algebraic group, and $\mathcal{U}$ a connected unipotent group.  Suppose we have an exact sequence of algebraic groups
$$1 \rightarrow \mathcal{U} \xrightarrow{i} H \xrightarrow{\pi} \mathcal{G} \rightarrow 1$$
such that $\pi$ induces an isomorphism of algebraic groups $H/\mathcal{U} \cong \mathcal{G}$.  This is referred to as being a \textit{strictly exact} sequence.  Since $k$ is algebraically closed and $\mathcal{U}$ is connected unipotent, there always exists an algebraic section $j: \mathcal{G} \rightarrow H$ such that $\pi \circ j$ is the identity morphism on $\mathcal{G}$ (see \cite[2.2]{M}).  Based on Theorem \ref{sectiontolift}, we need to know when we may choose an algebraic section which is a morphism of algebraic groups.  If such a morphism exists then we say the extension is split.

Suppose first that $\mathcal{U}$ is a vector group (isomorphic to $\mathbb{G}_a^{\times n}$ for some $n$).  The following is detailed in \cite[Section 4.4]{M}.  The conjugation action of $H$ on $\mathcal{U}$ factors through $\mathcal{G}$, so that $\mathcal{G}$ acts via group automorphisms on $\mathcal{U}$.  The section $j$ gives rise to a variety morphism
$$\beta_j: \mathcal{G} \times \mathcal{G} \rightarrow \mathcal{U}, \quad \beta_j(g_1,g_2) = j(g_2)^{-1}j(g_1)^{-1}j(g_1g_2).$$
If the action of $\mathcal{G}$ on $\mathcal{U}$ is $k$-linear (so that $\mathcal{U}$ is $\mathcal{G}$-equivariantly isomorphic to a rational $\mathcal{G}$-module $V$), then $\beta_j$ is a $2$-cocycle in the Hochschild complex for the $\mathcal{G}$-module $\mathcal{U}\cong V$.  The cohomology class $[\beta_j] \in H^2(\mathcal{G},\mathcal{U})$ is independent of the choice of the section $j$, and we obtain a bijection between equivalence classes of strictly exact extensions of $\mathcal{G}$ by $\mathcal{U}$ for which the action of $\mathcal{G}$ on $\mathcal{U}$ is $k$-linear and the group $H^2(\mathcal{G},\mathcal{U})$.  The extension is split if and only if $[\beta_j]=0$.

If the extension is split, then there is a complement to $\mathcal{U}$ in $H$, so that $H \cong \mathcal{U} \rtimes \mathcal{G}$.  An implication of \cite[4.5.2]{M} is that there is a single $H$-conjugacy class of such complements if $H^1(\mathcal{G},\mathcal{U})=0$.

These results can be applied to extensions of $\mathcal{G}$ by $\mathcal{U}$ when the latter is not a vector group which is $\mathcal{G}$-equivariantly isomorphic to a rational $\mathcal{G}$-module if one can find an algebraic group filtration of $\mathcal{U}$ with quotients which meet this criteria.  Independently, McNinch \cite[Theorem B]{M2} and Stewart \cite[Theorem 3.3.5]{St} have shown that this is always possible given the assumptions in our paper.

We will now recall this result for the strictly exact extension
$$1 \rightarrow U_A \rightarrow G^{\diamond} \xrightarrow{\pi} G \rightarrow 1$$
defined above.  Here we can formulate a slightly stronger statement thanks to the $G^{\diamond}$-equivariant isomorphism between $U_A = 1 + J_A$ and $\text{Lie}(U_A)=J_A$.

\begin{prop}\label{filtration}
For any algebraic subgroup $H \le G$, there is a $U_A$-central series $U_A=U_1 \ge U_2 \ge \cdots \ge U_n = {1}$ of connected closed subgroups, with each $U_i$ normal in $\pi^{-1}(H)$, such that that $U_i/U_{i+1}$ is a vector group which is $H$-equivariantly isomorphic to a simple rational $H$-module $V_i$.   
\end{prop}

\begin{proof}
The argument here is essentially the same as that used in the first half of the proof of \cite[3.3.5]{St}.  The map
$$J_A \rightarrow U_A, \quad X \mapsto 1+X$$
is a $G^{\diamond}$-equivariant isomorphism of varieties.  Since $G^{\diamond}$ acts on $A = \text{End}_{G_{(1)}}(Q_1(\lambda))$ by algebra automorphisms and $J_A$ is the radical of $A$, the powers of $J_A$ are stable under the action of $G^{\diamond}$.  Let $U_i = 1+(J_A)^i$.  This is a closed connected subgroup of $U_A$ which is normal in $G^{\diamond}$.  Moreover, we have that $U_i \le Z(U_A/U_{i+1})$, since if $X \in (J_A)^i$ and $Y \in J_A$, then
$$(1 + X)(1+Y) = (1+Y)(1+X)$$
modulo $U_{i+1}$.  We immediately see then that $U_A$ acts trivially on each quotient $U_i/U_{i+1}$ so that the isomorphism above induces $G$-equivariant variety isomorphisms
$$(J_A)^i/(J_A)^{i+1} \rightarrow U_i/U_{i+1}.$$
As the $G$-module on the left can be given a filtration by simple modules, it is clear that we can further refine the filtration $\{U_i\}$ so that $U_i/U_{i+1}$ is a $G$-equivariantly isomorphic to a simple $G$-module.

Using these arguments for the subgroup $\pi^{-1}(H) \le G^{\diamond}$ acting on $U_A$ then proves the full statement.
\end{proof}

The following propositions were established by McNinch using inductive arguments for arbitrary connected reductive groups, and appropriate group extensions.  We recall them here only in our particular setting.

First, let $H \le G$ and $U_i, V_i$ be as in the previous proposition.

\begin{prop}\label{splitting}\cite[5.1,5.2]{M}
If $H^2(H,V_i)=0$ for all $i$, then there is a morphism of algebraic groups $j_H:H \rightarrow G^{\diamond}$ such that $\pi \circ j_H$ is the identity morphism on $H$.  If $H^1(H,V_i)=0$ for all $i$, then $j_H(H)$ is unique up to conjugation in $G^{\diamond}$.
\end{prop}

Now choose $U_i$ relative to $G$.  Using the isomorphism of cohomology groups $H^2(G,V_i) \cong H^2(B,V_i)$, which is induced by the restriction of the Hochschild complex \cite[Proposition 4.2]{M}, we have:

\begin{prop}\label{borel}\cite[5.4]{M}
There is an algebraic group homomorphism $j_G:G \rightarrow G^{\diamond}$ splitting $\pi$ if and only if there is an algebraic group homomorphism $j_B:B \rightarrow \pi^{-1}(B)$ splitting the resulting extension of $B$ by $U_A$.
\end{prop}

\subsection{Second cohomology of $B$}

We require the following computations for $B$-cohomology given by Bendel, Nakano, and Pillen \cite{BNP}, and also obtained by Andersen and Rian \cite{AR} when $p>h$.  As we will often apply these results to Borel subgroups of Levi subgroups of $G$, we note that these formulas hold for arbitrary reductive groups for which the assumptions on the prime hold.

\begin{thm}\label{2ndcoh}\cite[Theorem 5.8]{BNP} Let $B$ be the Borel subgroup of a connected reductive group, and assume that $p > 3$.  Then
$$H^2(B,\lambda) = \begin{cases} k & \text{ if } \lambda = -p^i w \cdot 0 \text{ for } i \ge 0, \ell(w)=2\\
                                                          k & \text{ if } \lambda = p^i \alpha, \quad \alpha \in \Pi\\
                                                          k & \text{ if } \lambda = p^i\alpha - p^j\beta, \quad \alpha,\beta \in \Pi\\
                                                          0 & \textup{otherwise}\\
                                     \end{cases}$$
\end{thm}

\begin{remark}
We note that our signs differ from those in the cited papers due to our choice of $B$ corresponding to the positive roots rather than the negative ones.
\end{remark}

\bigskip


\section{Further Results}

We continue to work with the algebraic group extension
$$1 \rightarrow U_A \rightarrow G^{\diamond} \xrightarrow{\pi} G \rightarrow 1$$
from Section 2.3 associated to $Q_1(\lambda)$ for some fixed $\lambda \in X_1(T)$.

We first look at when an algebraic group morphism from a closed subgroup $H \le G$ to $G^{\diamond}$ leads to a lifting of $Q_1(\lambda)$ to $G_{(1)}H$.  In formulating this, we find it convenient to make use of the equivalence between first Frobenius kernels and restricted Lie algebras, noting that the section $j_{G_{(1)}}:G_{(1)} \rightarrow G^{\diamond}$ can also be viewed as a splitting on the level of Lie algebras, so that $$\mathfrak{g}^{\diamond} \cong dj_{G_{(1)}}(\mathfrak{g}) \oplus \textup{Lie}(U_A).$$
We now have the following:

\begin{prop}\label{extend}
Let $H$ be an algebraic subgroup of $G$.  If there is a morphism of algebraic groups $j_H: H \rightarrow G^{\diamond}$ such that $\pi \circ j_H$ is the identity morphism on $H$, and $dj_H(\mathfrak{h}) \subseteq dj_{G_{(1)}}(\mathfrak{g})$, then $Q_1(\lambda)$ lifts to $G_{(1)}H$. 
\end{prop}

\begin{proof}
The subgroup scheme $G_{(1)}H \le G$ is isomorphic to the quotient of the external semi-direct product $G_{(1)} \rtimes H$ by $G_{(1)} \cap H$ \cite[I.6.2]{J1}.  The conjugation action of $G^{\diamond}$ on $j_{G_{(1)}}(G_{(1)})$ factors via $\pi$ since $U_A$ acts trivally, therefore $j_H(H)$ acts on $j_{G_{(1)}}(G_{(1)})$ as $H$ acts on $G_{(1)}$.  We then get a morphism
$$G_{(1)} \rtimes H \xrightarrow{j_{G_{(1)}} \times j_H} G^{\diamond}.$$
Since $d\pi \circ dj_H$ is the identity on $\mathfrak{h}$ and $d\pi \circ dj_{G_{(1)}}$ is the identity on $\mathfrak{g}$, if $dj_H(\mathfrak{h}) \subseteq dj_{G_{(1)}}(\mathfrak{g})$ we must have that $dj_H=dj_{G_{(1)}}$ on $\mathfrak{h}$.  From this it follows that $j_H=j_{G_{(1)}}$ on $H_{(1)} = G_{(1)} \cap H$.  So the kernel of the morphism $j_{G_{(1)}} \times j_H$ is $G_{(1)} \cap H$, hence it gives a morphism from $G_{(1)}H$ to $G^{\diamond}$ extending $j_{G_{(1)}}$.  Composing with $\sigma^{\diamond}$ then gives the representation of $G_{(1)}H$ on $Q_1(\lambda)$.
\end{proof}

Recall that the morphism $\sigma: G \rightarrow G^*$ in Section 2.3 can be chosen so that $\sigma \mid_T$ is a group homomorphism.  In order to prove that several important subgroups of $G$ satisfy the assumptions above, and for other uses, we find it necessary to analyze the $T$-weights of $\text{Lie}(U_A) = J_A$.

\begin{prop}\label{weights}
Let $\mu$ be a weight of the $T$-module $\textup{End}_{G_{(1)}}(\widehat{Q}_1(\lambda))$.  Then:
\begin{enumerate}
\item $\mu \in pX(T) \cap \mathbb{Z}\Phi$.
\item $-w(\mu) \le 2(p-1)\rho$ for all $w \in W$.
\end{enumerate}

\end{prop}

\begin{proof}
Since $T_{(1)}$ acts trivially, we see that $\mu \in pX(T)$, so we may write $\mu = p\mu^{\prime}$ for some $\mu^{\prime} \in X(T)$.  From the proof of \cite[Lemma II.9.4]{J1} we have that
$$\text{End}_{G_{(1)}}(\widehat{Q}_1(\lambda))_{p\mu^{\prime}} = \text{Hom}_{G_{(1)}T}(\widehat{Q}_1(\lambda + p\mu^{\prime}),\widehat{Q}_1(\lambda)).$$
It follows that $p\mu^{\prime}$ is a weight only if $\widehat{L}_1(\lambda + p\mu^{\prime})$ is a $G_{(1)}T$-composition factor of $\widehat{Q}_1(\lambda)$.  Therefore $\lambda + p\mu^{\prime} \ge \lambda - 2(p-1)\rho$, the lowest weight of $\widehat{Q}_1(\lambda)$.  So $-\mu \le 2(p-1)\rho$, and $\mu \in \mathbb{Z}\Phi$.

Let $w \in W$, and $g \in N_G(T)$ with $g \mapsto w$.  For a $G_{(1)}T$-module $M$, let $^gM$ be the twist by $g$.  We have that
\begin{eqnarray*} \text{Hom}_{G_{(1)}T}(\widehat{Q}_1(\lambda + p\mu^{\prime}),\widehat{Q}_1(\lambda)) & \cong & \text{Hom}_{G_{(1)}T}(^g\widehat{Q}_1(\lambda + p\mu^{\prime}),^g\widehat{Q}_1(\lambda)) \\
& \cong & \text{Hom}_{G_{(1)}T}(\widehat{Q}_1(\lambda + w(p\mu^{\prime})),\widehat{Q}_1(\lambda)) \end{eqnarray*}
by \cite[Lemma II.11.7(b)]{J1} and the isomorphism $\widehat{Q}_1(\lambda + p\mu^{\prime}) \cong \widehat{Q}_1(\lambda) \otimes p\mu^{\prime}$.
Thus $w(\mu)$ is also a weight, and the inequality $-w(\mu) \le 2(p-1)\rho$ follows as above.
\end{proof}

\begin{prop}\label{rootgen}
Let $p>2$ and $H \le G$ be a subgroup generated by $T$ and a collection of root subgroups of $G$.  Then any homomorphism $j_H: H \rightarrow G^{\diamond}$ for which $\pi \circ j_H$ is the identity on $H$ must satisfy $dj_H(\mathfrak{h}) \subseteq dj_{G_{(1)}}(\mathfrak{g})$.
\end{prop}

\begin{proof}
Since $p>2$, we have that $\alpha \notin pX(T)$ for all $\alpha \in \Phi$.  Since $dj_H(\text{Lie}(U_{\alpha}))$ is a $j_H(T)$-weight space of weight $\alpha$, it must be contained in $dj_{G_{(1)}}(\mathfrak{g})$ since the weights of $\text{Lie}(U_A)$ are all in $pX(T)$.  Also, the only homomorphism of restricted Lie algebras from $\mathfrak{t}$ to $\text{Lie}(U_A)$ is the trivial homomorphism, hence we also get that $dj_H(\mathfrak{t}) \subseteq dj_{G_{(1)}}(\mathfrak{g})$.  Since $\mathfrak{h}$ is generated by these root subalgebras and $\mathfrak{t}$, this proves the statement.
\end{proof}


\section{Group Extensions of $B$}

In this section we study extensions of $B$, with an emphasis on extensions by $\mathbb{G}_a$.  However, we first observe that Proposition \ref{borel} can be strenghtened by replacing $B$ with $U$, as follows from the next result.

\begin{prop}\label{solvable}
Let $\mathcal{G}$ be a connected linear algebraic group, and consider a strictly exact sequence of algebraic groups
$$1 \rightarrow V \rightarrow H \xrightarrow{\pi} \mathcal{G} \rightarrow 1$$
such that $V$ is a vector group which is a rational $\mathcal{G}$-module under the conjugation action of $H$ on $V$.  Let $\mathcal{G}_1 \le \mathcal{G}$ be a connected closed subgroup for which the resulting chain map on Hochschild complexes induces an injective map $H^2(\mathcal{G},V) \rightarrow H^2(\mathcal{G}_1,V)$.  Then the sequence is split if and only if the resulting sequence for $\mathcal{G}_1$ is split.

In particular, this holds if $\mathcal{G}$ is solvable and $\mathcal{G}_1 = R_u(\mathcal{G})$.
\end{prop}

\begin{proof}
The first statement is proved by following part of the argument used in the proof of \cite[5.4.1]{M}.  Let $j: \mathcal{G} \rightarrow H$ be an algebraic section to $\pi$, and the $2$-cocycle $\beta_j: \mathcal{G} \times \mathcal{G} \rightarrow V$ be as in Section 2.4.  Restricting $\beta_j$ to $\mathcal{G}_1 \times \mathcal{G}_1$, we have that $[{\beta_j}\mid _{\mathcal{G}_1 \times \mathcal{G}_1}]=0$ if and only if the sequence
$$1 \rightarrow V \rightarrow \pi^{-1}(\mathcal{G}_1) \xrightarrow{\pi} \mathcal{G}_1 \rightarrow 1$$
is split.  But we are assuming that $[{\beta_j}\mid _{\mathcal{G}_1 \times \mathcal{G}_1}]=0$ implies $[\beta_j]=0$, so the result follows.

If $\mathcal{G}$ is solvable, then there are isomorphisms
$$H^n(\mathcal{G},V) \xrightarrow{\sim} H^n(R_u(\mathcal{G}),V)^{\mathcal{G}/R_u(\mathcal{G})}$$
which come via the natural restriction maps on cohomology
$$H^n(\mathcal{G},V) \rightarrow H^n(R_u(\mathcal{G}),V)$$
as shown in Proposition XI.10.2 and Lemma XI.9.1 of \cite{Mac}.  But these restriction maps are induced by the chain map between the respective Hochschild complexes, as given in equation (9.4) of \cite[XI.9]{Mac}. 
\end{proof}

\begin{thm}
Let $p>2$.  If $Q_1(\lambda)$ lifts to $G_{(1)}U$, then it lifts to $G$.
\end{thm}

\begin{proof}
Suppose there is a lift to $G_{(1)}U$, and let
$$\phi: G_{(1)}U \rightarrow GL(Q_1(\lambda))$$
be the morphism defining the representation.  After conjugating if necessary, we may assume that on $G_{(1)}$, $\phi = \sigma^{\diamond} \circ j_{G_{(1)}}$.  For each $u \in U$, we have that $\phi(u)$ and $\sigma(u)$ act in the same way on $\phi(G_{(1)})$, which implies that
$$\phi(u)\sigma(u)^{-1} \in \text{End}_{G_{(1)}}(Q_1(\lambda))^{\times}.$$
But the group on the right is the product of $U_A$ and the group of invertible scalar matrices, so we then have
$$\phi(u) \in Z(GL(Q_1(\lambda))\cdot G^*$$
which by the unipotence of $\phi(u)$ implies that $\phi(u) \in G^*$, so that $\phi: U \rightarrow G^*$.  In view of Lemma \ref{kernel}, the homomorphism
$\sigma^{\diamond}$ must define an isomorphism between the unipotent radical of $(\sigma^{\diamond})^{-1}(\phi(U))$ and $\phi(U)$, so that we in fact get
a homomorphism from $U$ to $G^{\diamond}$ which splits the sequence
$$1 \rightarrow U_A \rightarrow \pi^{-1}(U) \xrightarrow{\pi} U \rightarrow 1.$$
By the Proposition \ref{borel}, together with the previous proposition, this implies that there is an algebraic group section $G \rightarrow G^{\diamond}$ to $\pi$.  By Propositions \ref{rootgen} and \ref{extend}, this section, composed with $\sigma^{\diamond}$, extends the $G_{(1)}$-structure on $Q_1(\lambda)$.
\end{proof}

Let $B^{\diamond} = \pi^{-1}(B)$.  By Proposition \ref{filtration}, we can choose a filtration $\{U_i\}$ of $U_A$ consisting of normal subgroups in $B^{\diamond}$ for which the quotients $U_i/U_{i+1}$ are connected $1$-dimensional unipotent groups (necessarily isomorphic to $\mathbb{G}_a$).  The action of $B$ on these quotients is by some character $\lambda:B \rightarrow \mathbb{G}_m$, which must satisfy the conditions in Proposition \ref{weights}.  The following technical lemma will be needed later.  Let $\pi_i$ denote the surjective morphism
$$\pi_i: B^{\diamond} \rightarrow B^{\diamond}/U_i.$$

\begin{lemma}\label{adjoint}
Let $p>2$, and let $H \le B$ be a subgroup generated by $T$ and root subgroups.  Suppose there exists an algebraic group homomorphism $f: H \rightarrow B^{\diamond}/U_i$ such that the composite homomorphism to $B$ (after further modding out by $(U_A/U_i)$) is the identity on $H$.  Then
$$\textup{Lie}(\pi_i^{-1}(f(H))) \cong \textup{Lie}(U_i) \oplus dj_{G_{(1)}}(\mathfrak{h})$$
and this decomposition is stable under the adjoint action of $\pi_i^{-1}(f(H))$. 
\end{lemma}

\begin{proof}
We have that
\begin{eqnarray*}
\text{Lie}(B^{\diamond}/U_i) & = & d\pi_i(\text{Lie}(U_A)) \oplus d\pi_i(dj_{G_{(1)}}(\mathfrak{b}))\\
& \cong & \text{Lie}(U_A)/\text{Lie}(U_i) \oplus dj_{G_{(1)}}(\mathfrak{b}).
\end{eqnarray*}
By the argument in Proposition \ref{rootgen} it follows that
$$df(\mathfrak{h}) = d\pi_i(dj_{G_{(1)}}(\mathfrak{h})).$$
Thus
\begin{eqnarray*}
\text{Lie}(\pi_i^{-1}(f(H))) & = & \textup{Lie}(U_i) \oplus dj_{G_{(1)}}(\mathfrak{h})\\
& = & \text{Lie}(\pi_i^{-1}(f(H))) \cap \text{Lie}(U_A) \oplus \text{Lie}(\pi_i^{-1}(f(H))) \cap dj_{G_{(1)}}(\mathfrak{g}).\\
\end{eqnarray*}
The stability under the adjoint action then follows from the stability of $\text{Lie}(U_A)$ and $dj_{G_{(1)}}(\mathfrak{g})$ under the adjoint action of $G^{\diamond}$.

\end{proof}

We now point out the usefulness of this lemma.  Suppose that $H$ is generated by $T$ and $U_{\alpha}$ for some $\alpha \in \Pi$, and suppose that there is a homomorphism $f:H \rightarrow B^{\diamond}/U_i$ as in the lemma.  We then obtain a strictly exact sequence
$$1 \rightarrow U_i/U_{i+1} \rightarrow \pi_i^{-1}(f(H))/U_{i+1} \rightarrow H \rightarrow 1$$
The unipotent radical of $\pi^{-1}(f(H))/U_{i+1}$ is then $2$-dimensional (being an extension of $U_{\alpha}$ by $U_i/U_{i+1} \cong \mathbb{G}_a$).  The classification of connected unipotent groups of dimension $2$ is known, and the previous lemma tells us in this case that
$$\text{Lie}(R_u(\pi^{-1}(f(H))/U_{i+1})) \cong \text{Lie}(\mathbb{G}_a) \oplus \text{Lie}(\mathbb{G}_a)$$
and that the adjoint action of $R_u(\pi^{-1}(f(H))/U_{i+1})$ is trivial (as it stabilizes both factors).  This then further limits the possible group structures for\\ $R_u(\pi^{-1}(f(H))/U_{i+1})$.


\section{Specific Applications to Lifting Problem}

We conclude by showing how the work in the preceding sections can be used to obtained liftings of the PIMs to infinitesimal thickenings of particular subgroups of $G$.  Before stating our first result, we recall an important general lemma which is stated and proved in the proof of \cite[Proposition II.11.16]{J1}.

\begin{lemma}\label{injectivequotients}
Let $\mathcal{G}$ be an affine group scheme over $k$, and $N$ a normal subgroup scheme.  Let $V_1$ be a $\mathcal{G}$-module which is injective over $N$, and $V_2$ an injective $\mathcal{G}/N$-module.  Then $V_1 \otimes V_2$ is an injective $\mathcal{G}$-module.
\end{lemma}

\begin{thm}\label{Wlifts}
Let $W^{\prime} \le W$ be such that $p \nmid |W^{\prime}|$, and let $N_{W^{\prime}} \le N_G(T)$ be the inverse image of $W^{\prime}$ under the surjection $N_G(T) \rightarrow W$.  Then $\widehat{Q}_1(\lambda)$ lifts to $G_{(1)}N_{W^{\prime}}$, and there is a unique such lift whose socle is isomorphic to the restriction of $L(\lambda)$ to $G_{(1)}N_{W^{\prime}}$.  In particular, if $p \nmid |W|$, $\widehat{Q}_1(\lambda)$ lifts to a $G_{(1)}N_G(T)$-module.
\end{thm}

\begin{proof}
Take a filtration $U_A = U_1 \ge \cdots \ge U_n$ relative to $G$ as in Proposition \ref{filtration}.  Now $N_{W^{\prime}}/T \cong W^{\prime}$, and since $p \nmid |W^{\prime}|$, $k$ is an injective $W^{\prime}$-module.  The previous lemma then implies that every $N_{W^{\prime}}$-module is injective.  Thus $H^j(N_{W^{\prime}},U_{i}/U_{i+1})=0$ for all $j>0$.  It follows from Proposition \ref{splitting} that there is a homomorphism $j_{N_{W^{\prime}}}: N_{W^{\prime}} \rightarrow G^{\diamond}$ such that $\pi \circ j_{N_{W^{\prime}}}$ is the identity on $N_{W^{\prime}}$, and that the image of $j_{N_{W^{\prime}}}$ is unique up to conjugation in $G^{\diamond}$.  The connected component of $N_{W^{\prime}}$ is $T$, hence its Lie algebra is just $\mathfrak{t}$.  We must have that $dj_{N_{W^{\prime}}}(\mathfrak{t}) \rightarrow \text{Lie}(U_A)$ is the zero map, so that $dj_{N_{W^{\prime}}}(\mathfrak{t}) \subseteq dj_{G_{(1)}}(\mathfrak{g})$.  By Proposition \ref{extend}, this implies that there is a lifting to $G_{(1)}N_{W^{\prime}}$, and that this lift must be unique among those that factor through $G^*$.  But we note that any lifting of $\widehat{Q}_1(\lambda)$ to $G_{(1)}N_{W^{\prime}}$ for which the $G_{(1)}N_{W^{\prime}}$-socle is isomorphic to the restriction of $L(\lambda)$ must be such that the image of $N_{W^{\prime}}$ is in $G^* \le GL(Q_1(\lambda))$ (by comparing the action on the socle, as in the argument used in the Section 2 of \cite{D1}).
\end{proof}

\begin{remark}
In fact, a stronger statement is true.  When the conditions in the proposition hold, Lemma \ref{injectivequotients} implies that the lift of $\widehat{Q}_1(\lambda)$ is injective over $G_{(1)}N_{W^{\prime}}$, meaning we can choose a decomposition of $T(2(p-1)\rho + w_0\lambda)$ as in Section 2.3 so that $Q_1(\lambda)$ is a $G_{(1)}N_{W^{\prime}}$-summand, and hence can choose $\sigma$ to be a group homomorphism when restricted to $N_{W^{\prime}}$.
\end{remark}

We conclude by showing that if $p \ge h-2$ and $G=SL_n$, then $\widehat{Q}_1(\lambda)$ extends to a $G_{(1)}L_I$-module when $I=\{\alpha\}$ for a single simple root.

\begin{thm}
Let $p > 3$, and suppose that there is some $\alpha \in \Pi$ such that $c\alpha$ is not a weight of $\textup{Lie}(U_A)$ for all $c \ge p^2$.  Then $\widehat{Q}_1(\lambda)$ lifts to a $G_{(1)}L_{\alpha}$-module.
\end{thm}

\begin{proof}
By Propositions \ref{rootgen} and \ref{extend}, we want to find an algebraic group homomorphism $j_{L_{\alpha}}: L_{\alpha} \rightarrow G^{\diamond}$ such that $\pi \circ j_{L_{\alpha}}$ is the identity on $L_{\alpha}$.  By Proposition \ref{borel}, it suffices to find such a homomorphism for $B \cap L_{\alpha} = TU_{\alpha}$.
Take a $B$-filtration $\{U_i\}$ of $U_A$ as in Proposition \ref{filtration}.  Inductively, our problem reduces to repeatedly splitting extensions of the form
$$1 \rightarrow U_i/U_{i+1} \rightarrow H \rightarrow TU_{\alpha} \rightarrow 1$$
where $U_i/U_{i+1} \cong \mathbb{G}_a$, and the $T$-weight of the conjugation action on this group is a weight of $\text{Lie}(U_A)$.  By Proposition \ref{solvable}, this sequence splits if the sequence
$$1 \rightarrow U_i/U_{i+1} \rightarrow R_u(H) \rightarrow U_{\alpha} \rightarrow 1$$
splits, which is equivalent to saying that $R_u(H) \cong \mathbb{G}_a \times \mathbb{G}_a$.

Applying Lemma \ref{adjoint} (and the discussion following it) we have that
$$\text{Lie}(R_u(H)) \cong \text{Lie}(\mathbb{G}_a) \oplus \text{Lie}(\mathbb{G}_a),$$
and that the adjoint action of $R_u(H)$ on $\text{Lie}(R_u(H))$ is trivial.  Let $\lambda$ be the $T$-weight of $U_i/U_{i+1}$.  By Theorem \ref{2ndcoh}, the extension of $TU_{\alpha}$ by $U_i/U_{i+1}$ is automatically split unless $\lambda$ is either of the form $p^i\alpha$, with $i > 0$, or else of the form $p^i (1+ p^j)\alpha$, with $i\ge 0$ and $j >0$.  But we are assuming that $c\alpha$ is not a weight of $\text{Lie}(U_A)$ when $c \ge p^2$, further limiting our cases to $\lambda = p\alpha$ or $\lambda = (p+1)\alpha$.  Since $\langle (p+1)\alpha,\alpha^{\vee}\rangle = 2(p+1)$, which is not divisible by $p$ when $p\ne2$, we find that $(p+1)\alpha \notin pX(T)$, which by Proposition \ref{weights} leaves us to only check that the extension splits when $\lambda = p\alpha$.

We claim that the non-split extension corresponding to this weight occurs when $R_u(H)$ is isomorphic to the Witt group $\mathcal{W}_2$.  Indeed, the extension is determined by a $2$-cocycle
$$\beta: U_{\alpha} \times U_{\alpha} \rightarrow U_i/U_{i+1},$$
which is a morphism $\mathbb{G}_a \times \mathbb{G}_a \rightarrow \mathbb{G}_a$, and as such is given by a polynomial in two variables.  The cohomology class representatives of $H^2(\mathbb{G}_a,\mathbb{G}_a)$ are given in \cite[Remark 2.1.6]{DFPS}, where these polynomials have a $k$-basis consisting of the polynomials:
$$\left(\sum_{\ell = 1}^{p-1} \frac{(p-1)!}{\ell!(p-\ell)!}x^{\ell}y^{p-\ell}\right)^{p^i}, \quad i \ge 0,$$
and
$$\left(xy^{p^j}\right)^{p^i}, \quad j > 0, \, i \ge 0.$$

We have that the isomorphism
$$H^2(U_{\alpha},U_i/U_{i+1})^T \cong H^2(TU_{\alpha},U_i/U_{i+1})$$
is induced by the chain map between Hochschild complexes (see the proof of Proposition \ref{solvable}).  It follows from \cite[I.6.7]{J1} (after identifying the groups in the Hochschild complex with morphisms from $U_{\alpha}^{\times n}$ to $U_i/U_{i+1}$) that the $T$-action on the $2$-cocycle $\beta$ is given by
$$(t.\beta)(u_1,u_2) = t.(\beta(t^{-1}u_1t,t^{-1}u_2t))$$
It follows that if $\lambda = p\alpha$, then the non-split extension of $U_{\alpha}$ by $U_i/U_{i+1}$ whose $2$-cocycle is fixed by $T$ is a scalar multiple of the polynomial 
$$\sum_{\ell = 1}^{p-1} \frac{(p-1)!}{\ell!(p-\ell)!}x^{\ell}y^{p-\ell}$$
thus is isomorphic to $\mathcal{W}_2$.  But $\text{Lie}(\mathcal{W}_2) \not\cong \text{Lie}(\mathbb{G}_a) \oplus \text{Lie}(\mathbb{G}_a)$, hence $R_u(H) \not\cong \mathcal{W}_2$, showing that $R_u(H) \cong \mathbb{G}_a \times \mathbb{G}_a$, which by the reductions above proves the theorem.
\end{proof}

\begin{thm}\label{SLnLevi}
If $G=SL_n$ and $p \ge h-2$, then $Q_1(\lambda)$ lifts to $G_{(1)}L_{\alpha}$ for every simple root $\alpha$.
\end{thm}

\begin{proof}
Let $T$ be the subgroup of diagonal matrices, $B$ the subgroup of upper-triangular matrices, and $\alpha_1$ the first simple root in the standard ordering.  It is easy to check that
$$2(p-1)\rho = (p-1)\sum_{\alpha \in \Phi^+} = (p-1)(h-1)\alpha_1 + \sum_{i=2}^{n-1} c_i\alpha_i.$$
If $p \ge h-1$, then
$$p^2 \ge p(h-1) > (p-1)(h-1).$$
But also, if $p=h-2$, we have
$$p^2 = (h-2)^2 > (h-3)(h-1) = (p-1)(h-1).$$
Therefore we see that $p^2\alpha_1 \not\le 2(p-1)\rho$, hence cannot be a weight of $\text{Lie}(U_A)$ by Proposition \ref{weights}.  Since every root is in the $W$-orbit of $\alpha_1$, and the weights of $\text{Lie}(U_A)$ are stable under $W$, we see that the same holds for every other root.  Applying the previous theorem then completes the proof.
\end{proof}

\begin{remark}
Note that if $G=SL_n$ and $p > h$, then Theorems \ref{Wlifts} and \ref{SLnLevi} show that $Q_1(\lambda)$ can be lifted both to $G_{(1)}N_G(T)$ and to $G_{(1)}L_{\alpha_1}$, and moreover these lifts can be made to be compatible, in the sense that both give lifts to $G_{(1)}(N_G(T) \cap L_{\alpha_1})$, and we can conjugate the images of these lifts in $GL(Q_1(\lambda))$ to agree on this subgroup (by the uniqueness part of Theorem \ref{Wlifts}).  As $U_{\alpha_1}$ and $N_G(T)$ generate $G$, this is more evidence suggesting that $Q_1(\lambda)$ can be lifted to $G$, at least when $p>h$.
\end{remark}

\end{document}